\documentclass[10pt]{amsart}

\usepackage{tikz}
\usepackage{mathrsfs}
\usepackage{amsthm}
\usepackage{ulem}
\usepackage{amsmath,amssymb}
\usepackage[all,poly,knot]{xy}
\usepackage{hyperref}
\usepackage[center]{titlesec} 
\usepackage{amsfonts}
\usepackage{color}
\titleformat{\section}{\centering\Large\bfseries}{\arabic{section}.}{1em}{}
\titleformat{\subsection}{\centering\large\bfseries}{\arabic{section}.\arabic{subsection}.}{1em}{}

\newtheorem{thm}{Theorem}[section]

\newtheorem{lem}[thm]{{Lemma}}

\newtheorem{rmk}[thm]{Remark}
\theoremstyle{definition}

\newtheorem{setup}[thm]{Setup}
\newtheorem*{acknowledgement*}{Acknowledgements}

\theoremstyle{remark}
\newtheorem*{rmk*}{Remark}

\newtheorem{defi}[thm]{Definition}
\theoremstyle{definition}

\numberwithin{equation}{section}

\newcommand{\bZ}{{\mathbb Z}}

\newcommand{\nc}{\newcommand}
\nc{\on}{\operatorname}

\nc{\Aut}  {{\on{\mathrm  {Aut}}}}
\nc{\End}  {{\on{\mathrm  {End}}}}
\nc{\Fil}  {{\on{\mathrm  {Fil}}}}
\nc{\Frac} {{\on{\mathrm  {Frac}}}}
\nc{\Gal}  {{\on{\mathrm  {Gal}}}}
\nc{\GL}   {{\on{\mathrm  {GL}}}}
\nc{\Gr}   {{\on{\mathrm  {Gr}}}}
\nc{\Hom}  {{\on{\mathrm  {Hom}}}}
\nc{\id}   {{\on{\mathrm  {id}}}}
\nc{\PGL}  {{\on{\mathrm  {PGL}}}}
\nc{\rank} {{\on{\mathrm  {rank}}}}
\nc{\rmd}  {{\on{\mathrm  {d}}}}
\nc{\Spec} {{\on{\mathrm  {Spec}}}}

\nc{\HDF}  {{\on{\mathcal  {HDF}}}}
\nc{\HIG}  {{\on{\mathcal  {HIG}}}}
\nc{\IC}   {{\on{\mathcal {IC}}}}
\nc{\MCF}  {{\on{\mathcal {MCF}}}}
\nc{\MCFa} {{\on{\mathcal {MCF}_{[0,a]}}}}
\nc{\MF}   {{\on{\mathcal {MF}}}}
\nc{\MFa}  {{\on{         \MF_{[0,a]}}}}
\nc{\MFaf} {{\on{         \MF_{[0,a],f}}}}
\nc{\MIC}  {{\on{\mathcal {MIC}}}}
\nc{\MICa} {{\on{\mathcal {MIC}_{[0,a]}}}}
\nc{\THDF} {{\on{\mathcal {THDF}}}}
\nc{\THDFa}{{\on{\mathcal {THDF}_{[0,a]}}}}
\nc{\TMF}  {{\on{\mathcal {TMF}}}}
\nc{\TMFa} {{\on{\TMF_{[0,a]}}}}
\nc{\TMFaf}{{\on{\TMF_{[0,a],f}}}}
\nc{\tMIC} {{\on{\widetilde{\mathcal{MIC}}}}}
\nc{\FIsoc} {{\on{\textbf{F-Isoc}}}}

\def\nilp{\mathrm{nilp}}

\numberwithin{equation}{thm}

\usepackage[top=1.5in, bottom=1.5in, left=1in, right=1in]{geometry}

\def\Rep{\mathrm{Rep}^{\mathrm{irr.crys}}_{\mathrm{GL}_r(\mathbb Z_{p^f})}(\Gamma)}
\def\RepGeo{\mathrm{Rep}^{\mathrm{irr.crys}}_{\mathrm{GL}_r(\mathbb Z_{p^f})}(\Gamma,\bar \Gamma)}
\def\MF{\mathrm{MF}_{\textbf{FF}}}

\def\MFFil{\mathrm{MF}_{\textbf{FdR}}}
\def\MFnabla{\mathrm{MF}_{\textbf{dR}}}
\def\MFphi{\mathrm{MF}_{\textbf{F-Cris}}}
\def\tMFphi{\mathrm{MF}^{\rm twist}_{\textbf{F-Cris}}}
\def\MFtheta{\mathrm{MF}_{\textbf{Higgs}}}
\def\tMFIsoc{\mathrm{MF}^{\rm twist}_{\textbf{F-Isoc}}}
\def\MFIsoc{\mathrm{MF}_{\textbf{F-isoc}}}

\begin{document}
\title{ Finiteness of logarithmic crystalline representations II} 
\begin{abstract} Let $K$ be an unramified $p$-adic local field and let $W$ be the ring of integers of $K$. Let $(X,S)/W$ be a smooth proper scheme together with a simple normal crossings divisor and fix positive integers $r$ and $f$. We show that the set of absolutely irreducible representations $\pi_1(X_{\bar K})\rightarrow \mathrm{GL}_r(\mathbb{Z}_{p^f})$ that come from log crystalline $\bZ_{p^f}$-local systems over $(X_K,S_K)$ of rank $r$ is finite. The proof uses $p$-adic nonabelian Hodge theory and a finiteness result due Abe/Lafforgue.
\end{abstract}
\author{Raju Krishnamoorthy}
\email{raju@uga.edu}  
\address{Department of Mathematics, University of Georgia, Athens, GA 30605, USA}
\author{Jinbang Yang}
\email{yjb@mail.ustc.edu.cn}
\address{Institut f\"ur Mathematik, Universit\"at Mainz, Mainz 55099, Germany}
\author{Kang Zuo}
\email{zuok@uni-mainz.de}
\address{Institut f\"ur Mathematik, Universit\"at Mainz, Mainz 55099, Germany}

\maketitle
\section{Introduction}
To state our main theorem, the following setup will be convenient.
\begin{setup}\label{setup}Let $r$ and $f$ be positive integers. Let $p$ be an odd prime and let $k$ be a finite field containing $\mathbb F_{p^f}$. Set $W:=W(k)$ to be the ring of Witt vectors of $k$ and $K:=\text{Frac}(W)$. Let $(X,S)/W$ be a smooth projective scheme together with a relative simple normal crossings divisor over $W$. Set $U:=X\backslash S$. Let $x$ be a $\bar K$ point of $U$. For a positive integer $n\geq 1$ and a scheme $T$ over $W$, the notation $T_n$ refers to the reduction of $T$ modulo $p^n$. 
\end{setup}

The following is our main result, in which the base point is suppressed.
\begin{thm}\label{mainThm} Notation as in Setup \ref{setup}. Then the following set
\[\left.\left\{\rho\colon \pi^{et}_1(U_{K}) \rightarrow \mathrm{GL}_r(\mathbb Z_{p^f}) \left| \begin{array}{l}
			\rho \text{ is log crystalline }\\
			\text{with HT weights in }[a,a+p-1]\text{ for some }a\in\mathbb{Z} \\
			\text{and } \rho^{\text{geo}}\colon \pi_1^{\text{\'et}}(U_{\bar K})\rightarrow \mathrm{GL}_r(\mathbb{Q}_{p^f}) \text{ absolutely irreducible.}\\
			\end{array}  \right.\right\}
\right/\begin{array}{l}
			\rho_1\sim\rho_2\text{ if }\\
			\rho_1|_{\pi_1^{\text{\'et}}(U_{\bar K})}\cong \rho_2|_{\pi_1^{\text{\'et}}(U_{\bar K})}\\
			\end{array}
			\]	
is finite.
\end{thm}
Note that $\rho_1\sim\rho_2$ in Theorem \ref{mainThm} if and only if there exists a character $\chi\colon \text{Gal}(\bar K/K)\rightarrow \mathbb{Z}^{\times}_{p^f}$ such that $\rho_1\cong \rho_2\otimes \chi$ because $\rho_1$ and $\rho_2$ are assumed to be geometrically absolutely irreducible.

Crystalline representations are a $p$-adic analog of polarized variations of Hodge structures. Therefore Theorem \ref{mainThm} is an arithmetic analog of a theorem of Deligne \cite{Del87}. See also the very recent work of Litt for a finiteness result in a different spirit \cite{Li18}.

\section{Preliminaries}
First of all, we reduce Theorem \ref{mainThm} to the case of curves.
\begin{lem}\label{lemma:lefschetz}Notation as in Setup \ref{setup}. Then there exists a smooth projective relative curve $C\subset X$ over $W$ that intersects $S$ transversely, with the property that $\pi_1(C_K\cap U_K)\rightarrow \pi_1(U_K)$ is surjective. Therefore, to prove Theorem \ref{mainThm}, it suffices to consider the case when $X/W$ has relative dimension 1.
\end{lem}
\begin{proof}
We claim that there exists a smooth ample relative divisor $D\subset X$ over $W$ that intersects $S$ transversely. Indeed, pick some ample line bundle $L$ on $X$; then for all $m\gg 0$, the map $H^0(X,L^m)\rightarrow H^0(X_1,L^m_1)$ is surjective. On the other hand, for $m\gg 0$, the vector space $ H^0(X_1,L^m_1)$ has a section $s_1$ whose zero locus $V(s_1)$ is smooth and intersects $S_1$ transversely by Poonen's Bertini theorem \cite[Theorem 1.3]{poonenbertini}. Take any lift $s\in H^0(X,L^m)$ of $s_1$; then the zero locus $V(s)$ is smooth over $W$ and intersects $S$ transversely. Finally, it is well known that the map on fundamental groups $\pi_1(D_K\cap U_K)\rightarrow \pi_1(U_K)$ is surjective because $D_K\subset X_K$ is ample and $D_K$ intersects $S_K$ transversely. Proceed by induction. 

Now, as $\pi_1(C_K\cap U_K)\rightarrow \pi_1(U_K)$ is surjective, it follows that to prove Theorem \ref{mainThm}, it suffices to prove it for the pair $(C,S\cap C)$, i.e., we may reduce to the case of curves.
\end{proof}
To a logarithmic crystalline representation $\rho\colon \pi_1(U_K)\rightarrow \textrm{GL}_r(\mathbb{Z}_{p^f})$, we may attach an overconvergent $F$-isocrystal.\footnote{For details, see \cite[Section 2]{KYZfiniteness1}.} We now show that a logarithmic crystalline representation being irreducible implies that the attached overconvergent $F$-isocrystal is also irreducible. While this is not strictly useful for the rest of the article, it seems to be of independent interest.
\begin{lem}\label{irrF-Isoc}
Notation as in Setup \ref{setup}. Let $\rho\colon \pi_1(U_K) \rightarrow \mathrm{GL}_r(\mathbb Z_{p^f})$ be a crystalline representation with associated logarithmic Fontaine-Faltings module $(V,\nabla,\Fil,\varphi,\iota)$. 

If $\rho_{\mathbb Q}$ is irreducible then the overconvergent $F$-isocrystal $(V,\nabla,\varphi,\iota)_{\mathbb Q}$ in $\FIsoc^{\dagger}(U_1)_{\mathbb{Q}_{p^f}}$ is irreducible.
\end{lem}

\begin{proof} First of all, it follows from \cite[Theorem 6.4.5]{kedlayasemistableI} that it suffices to check that $(V,\nabla,\varphi,\iota)_{\mathbb Q}$ is irreducible in $\FIsoc^{\nilp}_{\log}(X_1,S_1)_{\mathbb Q_{p^f}}$. Our proof will proceed by contradiction. 

Let $\Phi$ be a local lifting of the absolute Frobenius on $(X_1,S_1)$. For the logarithmic Fontaine-Faltings module, locally the $\varphi$-structure can be represented as an isomorphism
\[\varphi\colon \Phi^*\widetilde{(V,\nabla,\Fil)} \xrightarrow{\simeq} (V,\nabla),\]
where $\widetilde{(\cdot)}$ is Faltings' tilde functor. In the case when $V$ is $p$-torsion free, one may describe this as follows: 
\[\widetilde{(V,\nabla,\Fil)} = \sum_i \frac{\Fil^i(V,\nabla)}{p^i} \subset (V,\nabla,\Fil)_{\mathbb Q}.\]
That $\varphi$ is an isomorphism encodes the strong divisibility in the definition of a Fontaine-Faltings module. After shifting the filtration, we assume $\Fil^0V = V$. In this case, $(V,\nabla) \subset \widetilde{(V,\nabla,\Fil)}$. and $\varphi$ can be restricted on $\Phi^* (V,\nabla)$,
\[\varphi\colon \Phi^* (V,\nabla)\rightarrow (V,\nabla).\]
This yields the underlying logarithmic $F$-crystal.

 Suppose the $F$-isocrystal $(V,\nabla,\varphi,\iota)_{\mathbb Q}$ is not irreducible in $\FIsoc_{\log}^{\nilp}(X,S)_{\mathbb{Q}_{p^f}}$. Let $(\mathcal V',\nabla',\varphi',\iota')$ be a proper sub $F$-isocrystal of $(V,\nabla,\varphi,\iota)_{\mathbb Q}$ in $\FIsoc_{\log}^{\nilp}(X,S)_{\mathbb{Q}_{p^f}}$.  In particular, $\mathcal V'$ is stable under the $\mathbb{Q}_{p^f}$-action, the connection $\nabla$, and $\varphi$.
There is a natural choice of lattice: 
\[V' = \mathcal V'\cap V\]
Then the restriction of $\varphi$ induces a map
\begin{equation}\label{varphi'}
 \varphi'\colon \Phi^*(V',\nabla') \rightarrow (V',\nabla'),
\end{equation} 
since $\varphi(\Phi^*((V',\nabla'))) \subset \varphi(\Phi^*((\mathcal V',\nabla')) \cap \varphi(\Phi^*((V,\nabla))) \subset (\mathcal V',\nabla') \cap (V,\nabla) = (V',\nabla')$.
Denote $\Fil'$ the restriction of $\Fil$ on $V'$. The endomorphism structure clearly restricts on the quadruple $(V',\nabla',\Fil',\varphi')$. In the following, we show that $(V',\nabla',\Fil',\varphi')$ forms a sub-Fontaine-Faltings module of $(V,\nabla,\Fil,\varphi)$. As the triple $(V',\nabla', \varphi')$ is a logarithmic $F$-crystal in finite, locally free modules, we must check that the pair $(\Fil',\varphi')$ is strongly divisible, i.e., that the isogeny $\varphi\colon \Phi^*(V',\nabla')\rightarrow (V',\nabla')$ extends to an isomorphism
\[\Phi^*\widetilde{(V',\nabla',\Fil')}\rightarrow (V',\nabla')\]
Denote 
\[(V'',\nabla'',\Fil'') := (V,\nabla,\Fil)/(V',\nabla',\Fil').\]
We constructed the embedding $(V',\nabla',\Fil')\hookrightarrow (V,\nabla,\Fil)$ to be saturated and strict with respect to the filtrations. Therefore the triple $(V'',\nabla'',\Fil'')$ is a filtered logarithmic de Rham bundle.

Applying Faltings' tilde functor, one has short exact sequence
\[0\rightarrow \widetilde{V'} 
\longrightarrow \widetilde{V}  
\longrightarrow \widetilde{V''} \rightarrow 0.\]
Locally, one has the following commutative diagram
\begin{equation*}
\xymatrix{
0 \ar[r]  & 
\Phi^*\widetilde{V'}  \ar[r] \ar[d]^{\varphi'} &
\Phi^*\widetilde{V}   \ar[r] \ar[d]_\cong^{\varphi} &
\Phi^*\widetilde{V''} \ar[r] \ar@{..>}[d]^{\exists\varphi''} & 0  \\
0 \ar[r] & 
V' \ar[r] &
V \ar[r] &
V'' \ar[r] & 0  \\
}
\end{equation*}
where $\varphi'$ (by abusing notation as in (\ref{varphi'})) is the restriction of $\varphi$ on $\Phi^*\widetilde{V'}$, which extends the $\varphi'$ in (\ref{varphi'}). The image $\varphi'(\Phi^*\widetilde{V'})$ is contained in $V'$, because 
\[\varphi(\Phi^*(\widetilde{V'})) \subset \varphi(\Phi^*( \mathcal V') \cap \varphi(\Phi^*(\widetilde{V})) \subset \mathcal V'\cap V  = V'.\]
Since $\varphi$ is surjective,  $\varphi''$ is also surjective. On the other hand $\Phi^*\widetilde{V''}$ and $V''$ are bundles with the same rank, so $\varphi''$ is actually an isomorphism. By the snake lemma $\varphi'$ is also an isomorphism. This proves the strong divisibility, as desired.

By the Fontaine-Lafaille-Faltings correspondence, to $(V',\nabla',\Fil',\varphi',\iota')$ one may attach a (log crystalline) subrepresentation $\rho'$ of $\rho$ of strictly smaller rank \cite[Theorem 2.6* (i)]{Fal89}. It follows that $\rho_{\mathbb Q}$ is not irreducible, contradicting our hypothesis.
\end{proof}

The following is a version of Lemma \ref{irrF-Isoc} with the stronger hypothesis that $\rho_{\mathbb Q}$ is \emph{geometrically absolutely irreducible}. With these assumptions, we show something much stronger than the conclusion of Lemma \ref{irrF-Isoc}. This will be essential in the proof of Theorem \ref{mainThm}.

\begin{lem}\label{absolutely_irreducible}Notation as in Setup \ref{setup} and suppose $X/W$ is a curve. Let $\rho\colon \pi_1(U_K)\rightarrow \textrm{GL}_r(\mathbb{Z}_{p^f})$ be a logarithmic crystalline representation such that $\rho_{\mathbb Q}$ is geometrically absolutely irreducible. Let $(V,\nabla,\Fil, \Phi,\iota)$ be the associated logarithmic Fontaine-Faltings module with endomorphism structure $\iota$. Let $(V,\nabla)^{(0)}$ be the identity eigenspace of the $\iota$-action. Then the following holds.
\begin{enumerate}
\item The logarithmic de Rham bundle $(V,\nabla)^{(0)}_{{\mathbb Q}^{\rm unr}_p}$ on $(X,S)_{{\mathbb Q}_p^{\rm unr}}$ admits no proper de Rham subbundle.
\item the  object $(V,\nabla,\Phi,\iota)_{\mathbb Q}\in \FIsoc^{\dagger}(U_k)_{\mathbb{Q}_{p^f}}$ is absolutely irreducible.
\end{enumerate}
\end{lem}

\begin{proof}
We first prove the first statement. Fist of all, $(V,\nabla)^{(0)}_{K}$ is a semistable de Rham bundle of degree 0. Suppose for contradiction that there is a logarithmic de Rham subbundle $(V',\nabla')_{K'}$ of $(V,\nabla)_{K'}$ for some finite unramified extension $K'/K$. Set $\mathcal O'$ to be the ring of integers and $k'$ to be the residue field. Then $(V',\nabla')_{K'}$ automatically has nilpotent residues and hence also has degree 0 and is therefore semistable. We claim that we may find a logarithmic de Rham subbundle $(W,\nabla)$ of $(V,\nabla)^{(0)}_{\mathcal O'}$ over $(X,S)_{\mathcal O'}$ such that $(W,\nabla)_{K'}\cong (V',\nabla')_{K'}$ and $(W,\nabla)_{k'}$ is semistable of degree 0. First of all, there is clearly an extension to a torsion-free logarithmic de Rham subsheaf $(W,\nabla)$. By \cite[Section 6]{KYZlefschetz}, the degree of $W_{k'}$ is 0; therefore $(W,\nabla)_{k'}$ is a degree 0 logarithmic de Rham subsheaf of $(V,\nabla)_{k'}$. Note that $(W,\nabla)_{k'}\subset (V,\nabla)^{(0)}_{k'}$; as the latter is semistable of degree 0, it follows that the inclusion $(W,\nabla)_{k'}\subset (V,\nabla)^{(0)}_{k'}$ is saturated and hence is a de Rham subbundle.\footnote{Note that because $X_k$ is a smooth curve, any torsion-free sheaf is automatically a vector bundle.} As both de Rham bundles have degree 0 and $(V,\nabla)_{k'}$ is semistable, it follows that $(W,\nabla)_{k'}$ is semistable.

Let $\HDF$ be the (logarithmic) Higgs-de Rham flow attached to $\rho$. Run the Higgs-de Rham flow over $X$ with initial term $(W,\nabla)$, where the Hodge filtrations are chosen to be the restrictions of $\Fil$ on $V$. One obtains a sub Higgs-de Rham flow $\HDF'$ of $\HDF$. In general, this sub Higgs-de Rham is {\it not preperiodic}. Nonetheless, we claim that $\HDF'$ is preperiodic over each truncated level $W_m(k')$. This holds because $(V,\nabla)_{W_m(k')}$ has only finitely many subbundles of degree $0$ for each $m$.

Note that because there are only finitely many (isomorphism classes of) Higgs terms in $\HDF$ by the periodicity. Therefore we may inductively shift the index of $\HDF'$, to find a sequence of sub Higgs-de Rham flows $\HDF'_{W_m(k')}\subset \HDF_{W_m(k')}$ which are periodic with periodicity $f_m$, and satisfying $\HDF'_{W_{m+1}(k')}\equiv \HDF'_{W_m(k')} \pmod{p^m}$ and $f_m\mid f_{m+1}$.

To each of these truncated periodic Higgs-de Rham flows, there is an associated torsion logarithmic crystalline representation
\[\rho'_m\colon \pi_1(U_{\mathbb Q_p^{\rm unr}}) \rightarrow GL_s(W_{m}(\overline{k})).\]
Recall that $\widehat{\mathcal O_{\mathbb Q_p^{\rm unr}}}$, the $p$-adic completion of $\mathcal O_{\mathbb Q_p^{\rm unr}}$, is equal to $W(\overline{k})$.
Taking the inverse limit over $m$, one obtains a sub-representation
\[\rho'\colon \pi_1(U_{\mathbb Q_p^{\rm unr}}) \rightarrow GL_s(\widehat{\mathcal O_{\mathbb Q_p^{\rm unr}}})\]
of $\rho\colon \pi_1(U_{\mathbb Q_p^{\rm unr}}) \rightarrow GL_r(\widehat{\mathcal O_{\mathbb Q_p^{\rm unr}}})$. We claim this is in contradiction with the fact that $\rho\otimes {\mathbb Q_p}$ is geometrically absolutely irreducible. Indeed,$\rho'\mid_{\pi_1(U_{\bar K})}\otimes {\mathbb C_p}$ is a non-trivial sub-representation of $\rho\mid_{\pi_1(U_{\bar K})}\otimes {\mathbb C_p}$; on the other hand, the fact that  $\rho\mid_{\pi_1(U_{\bar K})}\otimes {\overline{\mathbb Q}_p}$ is irreducible implies that $\rho\mid_{\pi_1(U_{\bar K})}\otimes {\mathbb C_p}$ is also irreducible.
\end{proof}

We come to the following crucial definition.
\begin{defi} Notation as in Setup \ref{setup}. Let $(\mathcal{V},\nabla,\varphi,\iota)$ be an object of $\FIsoc^{\nilp}_{\log}(X_1,S_1)_{\mathbb Q_{p^f}}$. An \emph{extension} of $(\mathcal{V},\nabla,\varphi,\iota)$ is an logarithmic $F$-crystal in finite, locally free modules $(V,\nabla,\varphi,\iota)$ with $\mathbb Z_{p^f}$-structure such that $(V,\nabla,\varphi,\iota)_{\mathbb Q_p}\cong (\mathcal{V},\nabla,\varphi,\iota)$. An extension $(V,\nabla,\varphi,\iota)$ is said to be \emph{semistable} if the logarithmic flat connection $(V,\nabla)_1$ on $(X_1,S_1)$ is semistable.
\end{defi}

Recall that a rank-$1$ $F$-crystal over $k$ with $\mathbb Z_{p^f}$-structure is a pair $(L,\varphi)$ where $L$ is a finite free $W\otimes_{\mathbb Z_p} \mathbb Z_{p^f}$-module and $\varphi\colon L\rightarrow L$ is an injective $\sigma\otimes 1$-semilinear map where $\sigma\colon W\rightarrow W$ is the canonical lift of Frobenius. For any element $r\in \left(K\otimes_{\mathbb Q_p} \mathbb Q_{p^f}\right)^\times \cap (W\otimes_{\mathbb Z_p}\mathbb Z_{p^f})$, we denote 
$L_r = W\otimes_{\mathbb Z_p} \mathbb Z_{p^f}\cdot e$ with $\varphi_{L_r}(e)=re$.
 Conversely, for any rank-$1$ $F$-crystal over $W$ with $\mathbb Z_{p^f}$-structure is isomorphic to some $L_r$.

By tensoring $\mathbb Q_p$, one gets an rank-1 $F$-isocrystal $\mathcal L_r=L_r\otimes_{\mathbb Z_p}\mathbb Q_p$ over $k$ with $\mathbb Q_{p^f}$-structure.

Let $(V,\nabla,\varphi,\iota)$ be a logarithmic $F$-crystal in finite, locally free modules over $(X_1,S_1)$ with $\mathbb Z_{p^f}$-structure. Locally one can view $\varphi$ as a $\Phi\otimes 1$-semilinear map of the $\mathcal O_X\otimes \mathbb Z_{p^f}$-modules. We define the twist of $(V,\nabla,\varphi,\iota)$ by $L_r$ to be: $(V,\nabla,\varphi,\iota)\otimes L_r=(V,\nabla,r\cdot\varphi,\iota)$.

\begin{rmk}\label{InvariantUnderlyingdRBundle}
	Twisting by a constant rank 1 object does not change the underlying de Rham bundle.
\end{rmk}

\begin{lem}\label{BijectionExtensions}
	Let $(\mathcal{V},\nabla,\varphi,\iota)$ be an object of $\FIsoc^{\nilp}_{\log}(X_1,S_1)_{\mathbb Q_{p^f}}$. Let $L$ be a rank-1 $F$-crystal over $k$ with $\mathbb Z_{p^f}$-structure and let $\mathcal L=L\otimes_{\mathbb Z_p}\mathbb Q_p$. Denote $\varphi'=\varphi\otimes\varphi_L$. Then tensoring $L$ induces an injection 
	\[\{\text{extension of } (\mathcal{V},\nabla,\varphi,\iota) \} \rightarrow
	  \{\text{extension of } (\mathcal{V},\nabla,\varphi',\iota)\} \]
\end{lem}

\begin{proof} Since an extension of an $F$-isocrystal is uniquely determined by the extension of the underlying de Rham bundle and twisting a constant rank-1 object doesn't change the underlying de Rham bundle, the map is injective.
\end{proof}

\begin{lem}\label{finitenessExtension} Notation as in Setup \ref{setup}, and suppose $X/W$ is a curve. Let $(\mathcal V,\nabla,\varphi, \iota)$ be an irreducible object of $\FIsoc^{\nilp}_{\log}(X_1,S_1)_{\mathbb{Q}_{p^f}}$. Then there exists only finitely many isomorphism classes of semistable extensions $(V,\nabla,\varphi,\iota)$ of  $(\mathcal V,\nabla,\varphi, \iota)$.
\end{lem}

\begin{proof}
Assume there are infinitely many isomorphism classes of semistable extensions,  choose a representative from each isomorphism class, and enumerate them as follows $\{T_i=(V,\nabla,\varphi,\iota)_i\mid i\in I\}$. We will construct an infinite descending chain $T_{i_0}\supset T_{i_1}\supset T_{i_2}\dots$ such that the intersection yields a proper logarithmic sub-$F$-isocrystal, contradicting out original assumption.

Fix one element $i_0\in I$, and set $I_0=I$ and $J_0 = I_0-\{i_0\}$. By assumption, we may embed each $(V,\nabla,\varphi,\iota)$ as a lattice in $(\mathcal V,\nabla,\varphi, \iota)$. By multiplying by a suitable power of $p$ on each $T_i$, for $i\in J_0$, we may assume that 
\[T_i \subset T_{i_0} \quad \text{ and } \quad  T_i \not\subset pT_{i_0}.\]

This is equivalent to saying that the image of $V_i$ in $V_{i_0}/pV_{i_0}$ is a proper submodule. In fact, we claim that the image, namely $(V_i+pV_{i_0})/pV_{i_0}$, together with the induced logarithmic flat connection, is a semistable logarithmic de Rham bundle on $(X_1,S_1)$. Indeed, both $V_{i_0}/pV_{i_0}$ and $V_i/pV_i$ admit semistable flat connections of degree zero. Therefore the image $(V_i+pV_{i_0})/pT_{i_0}$ has degree 0 and hence, when equipped with the induced connection, is semistable. Finally, we claim that $(V_i+pV_{i_0})/pV_{i_0}$ is a subbundle of $V_{i_0}/pV_{i_0}$ (as opposed to merely a subsheaf). If not, the saturation would be a subbundle; but any non-trivial saturation increases the degree. As $(V_i+pV_{i_0})/pV_{i_0}$ with the induced flat connection is semi-stable, so is the saturation; this contradicts semistability of $V_{i_0}/pV_{i_0}$.

Consider the map  
\begin{equation*}
\xymatrix@R=0mm{
f_0\colon J_0 \ar[r] & 
\Sigma_0:= \{\text{ proper sub bundles of } V_{i_0}/p\cdot V_{i_0} \text{ of degree } 0\}\\
i \ar@{|->}[r] &  \Big(V_i+pV_{i_0}\Big)/pV_{i_0}
}
\end{equation*}
 The initial set is infinite by assumption. On the other hand, the terminal set is finite; indeed, this follows because the set all subbundles with fixed degree of a given bundle forms a bounded family and our base field is finite.

Thus there exists a submodule $\overline{M}_0$ of $V_{i_0}/pV_{i_0}$ such that $I_1:=f_0^{-1}(\overline{M}_0)$ is infinite. For any fixed $i\in I_1$, the submodule $V_i+pV_{i_0}$ is the inverse image of $\overline{M}_0$ under surjective map $V_{i_0} \rightarrow V_{i_0}/pV_{i_0}$; hence, the submodule $V_i+pV_{i_0}$ does not depend on the choice of $i \in I_1$. We further claim that for each $i\in I_0$, the module $V_i+pV_{i_0}$ together with the induced logarithmic flat connection, Frobenius structure, and endomorphism structure, yields a semistable extension of $(\mathcal V,\nabla,\varphi,\iota)$. This follows from the fact that $V_i/pV_i$ and $V_0/pV_0$, equipped with their flat connections, are semistable de Rham bundles of degree 0. Thus there exists $i_1\in I_1$ such that 
\[T_{i_1} = T_i+pT_{i_0} \text{ for all } i \in I_1.\]
Denote $J_1 = I_1\setminus\{i_1\}$. Then for all $i\in J_1$ one has
\[T_i \subsetneq T_{i_1}\subsetneq T_{i_0} \quad \text{ and } \quad  T_i \not\subset pT_{i_1}.\] 
Repeating the process, one can find a sequence of extensions
\[\cdots \subsetneq  T_{i_3}\subsetneq  T_{i_2} \subsetneq T_{i_1}\subsetneq  T_{i_0}\]
satisfying $T_{i_m}\not\subset pT_{i_n}$ for all $m,n\geq0$. 

Denote $T_{\infty} = \cap_{k=0}^{\infty} T_{i_k}$. In the following we show that $(T_{\infty})_{\mathbb Q_p}$ is a proper logarithmic sub $F$-isocrystal of $(\mathcal V,\nabla,\varphi,\iota)$. Thus we get a contradiction with the irreducibility of $(\mathcal V,\nabla,\varphi,\iota)\in \FIsoc^{\nilp}_{\log}(X_1,S_1)_{\mathbb Q_{p^f}}$.

Locally, we may assume $T_{i_k}$ are free modules of the same rank $r$ over a regular local ring $R$ satisfying 
\[\cdots \subsetneq  T_{i_3}\subsetneq  T_{i_2} \subsetneq T_{i_1}\subsetneq  T_{i_0}.\]
Since $T_{\infty} = \cap_{k=0}^{\infty} T_{i_k}$ is torsion-free and finitely generated over $R$, we may choose a free sub $R$-module of $T'_{\infty}\subseteq T_{\infty}$ with maximal rank $r_\infty\leq r$. We only need to show 
\[T_{\infty}\neq 0 \text{ and } r_\infty \neq  r.\]

We first show that $T_{\infty}\neq 0$. For a given positive integer $n$, consider the descending sequence 
\[\Bigg(\Big(T_{i_k} + p^nT_{i_0}\Big)/p^nT_{i_0}\Bigg)_k.\]
We claim the sequence stabilizes for $k\gg 0$. Each term is contained in $T_{i_0}/p^nT_{i_0}$.  Let's consider the index between $T_{i_0}$ and $p^nT_{i_0}$, which has only $p$-primary part and is finite. Thus the increasing sequence of the index $[T_{i_0}:T_{i_k}+p^nT_{i_0}]$ with upper bound $[T_{i_0}:p^nT_{i_0}]$ is stable. This implies $[T_{i_0}:T_{i_k}+p^nT_{i_0}] = [T_{i_0}:T_{i_{k+1}}+p^nT_{i_0}]$ for sufficiently large $k$, So $T_{i_k}+p^nT_{i_0} = T_{i_{k+1}}+p^nT_{i_0}$  for $k\gg 0$. Denote 
\[\overline{T}^{(n)}_0 := \bigcap_{k=0}^\infty \Big(T_{i_k} + p^nT_0\Big)/p^nT_0 = \Big(T_N + p^nT_0\Big)/p^nT_0\neq 0, \text{ for } N>>0.\]
Thus one has an surjective inverse system
\[\cdots \twoheadrightarrow \overline{T}^{(3)}_0 \twoheadrightarrow \overline{T}^{(2)}_0 \twoheadrightarrow \overline{T}^{(0)}_0 \twoheadrightarrow \overline{T}^{(0)}_0\]
whose inverse limit $\varprojlim\limits_{n} \overline{T}^{(n)}_0$ is non-empty. By the left exactness of inverse limits, the inclusion maps 
\[\left(\overline{T}^{(n)}_0 \subset \Big(T_{i_k} + p^nT_0\Big)/p^nT_0\right)_n\]
induce an injective map 
\[\varprojlim\limits_{n} \overline{T}^{(n)}_0 \hookrightarrow \varprojlim\limits_{n}\Big(T_{i_k} + p^nT_0\Big)/p^nT_0 = T_{i_k}.\]
Thus $\varprojlim\limits_{n} \overline{T}^{(n)}_0 \subset T_\infty = \bigcap_k T_{i_k}$. This implies that $T_\infty\neq0$.
The fact that $T_{\infty}\neq 0$ immediately implies that $T_{\infty}$ yields a logarithmic $F$-crystal in finite modules on $(X_1,S_1)$.

We now show that $r_{\infty}\neq r$. By \'etale localization, we reduce to the following setup in linear algebra. Let $A=W<x>$ be the $p$-adic completion of a polynomial ring in a single variable over $W$, and let $M_0\supsetneq M_1\supsetneq\dots$ be an infinite nested collection of finite free modules of fixed rank $r$, that is strictly decreasing and such that $M_j\not\subset p M_k$ for $j,k\geq 0$. Set $M_{\infty}:=\cap^{\infty}_{j=0} M_j$. We wish to prove that $M_{\infty}$ has rank smaller than $r$; equivalently, that it does not contain a lattice $L_{\infty}$ in $M_0\otimes \text{Frac}(A)$.  If it did, then $M_{\infty}$ would have finite, $p$-primary index in $M_0$. However, the index of $M_j$ in $M_0$ gets arbitrarily large; indeed, if $[M_j:M_k]=1$, then $M_j=M_k$. As index is multiplicative, we obtain a contradiction. 
\end{proof}

\begin{lem}\label{finiteHodgeFil} Notation as in Setup \ref{setup}. Let $(\mathcal{V},\nabla,\varphi,\iota)$ be an irreducible object of $\FIsoc^{\nilp}_{\log}(X_1,S_1)_{\mathbb Q_{p^f}}$. Let $(V,\nabla,\varphi,\iota)$ be an extension $(\mathcal{V},\nabla,\varphi,\iota)$. Then there exists only finitely many Hodge filtrations $\Fil_1$ on $(V,\nabla,\iota)_1 := (V,\nabla,\iota)\pmod{p}$ with $\Fil_1^0V_1=V_1$ and $\Fil_1^pV_1=0$ such that there exists $\varphi_1$ rendering the quintuple $(V,\nabla,\Fil,\varphi,\iota)_1$ a logarithmic Fontaine-Faltings module with endomorphism structure over $(X_1,S_1)$. 
\end{lem}

\begin{proof}
By the Fontaine-Lafaille-Faltings correspondence \cite[Theorem 2.6*(i)]{Fal89}, the category of $p$-torsion logarithmic Fontaine-Faltings modules (with endomorphism structure) on $(X_1,S_1)$ is equivalent to the category of logarithmic crystalline representations of $\pi^{\text{\'et}}_1(U_K)$ with coefficients in $\mathbb{F}_{p^f}$.
The \'etale fundamental group $\pi^{\text{\'et}}_1(U_K)$ is topologically finitely generated. Therefore the set of isomorphism classes of $\GL_r(\mathbb F_{p^f})$ representations of $\pi^{\text{\'et}}_1(U_K)$ is finite. In particular, there are only finitely many isomorphism classes of crystalline $\GL_r(\mathbb F_{p^f})$ representations. Forgetting the $\varphi$-structure, it follows that the set of isomorphism classes of de Rham bundles (with endomorphism structure) which underlie a Fontaine-Faltings module (with endomorphism structure) over $(X_1,S_1)$ is also finite.
	
Suppose there are infinitely many distinct Hodge filtrations $\Fil_1^{(i)}$ ($i=1,2,\cdots$) on $(V,\nabla,\iota)_1$ such that for each $i$, there exists $\varphi^{(i)}$ rendering the quintuple $(V,\nabla,\Fil,\varphi,\iota)_1$ a Fontaine-Faltings module. By the pigeonhole principle, there are infinitely $i$ such that there exists a log Fontaine-Faltings module  $(V,\nabla,\Fil^{(i)},\varphi^{(i)},\iota)_1$ whose isomorphism class is independent of $i$. In particular, one deduces $\mathrm{Aut}(V_1)$ is a infinite set. But this contradicts the finiteness of $\mathrm{Aut}(M)$ for any vector bundle $M$ over $X_1$, as our base field is finite.

\end{proof}

\begin{lem}\label{uniquenessliftFil}
Notation as in \ref{setup}. Let $(\mathcal{V},\nabla,\varphi,\iota)$ be an irreducible object of $\FIsoc^{\nilp}_{\log}(X_1,S_1)_{\mathbb Q_{p^f}}$. Let $(V,\nabla,\varphi,\iota)$ be an extension $(\mathcal{V},\nabla,\varphi,\iota)$. Let $\Fil_1$ be a Hodge filtration on $(V,\nabla,\iota)_1 := (V,\nabla,\iota)\pmod{p}$ with $\Fil_1^0V_1=V_1$ and $\Fil_1^pV_1=0$ such that there exists $\varphi_1$ rendering the quintuple $(V,\nabla,\Fil,\varphi,\iota)_1$ a logarithmic Fontaine-Faltings module with endomorphism structure over $X_1$. Assume there exists two liftings $\Fil$ and $\Fil'$ of the Hodge filtration $\Fil_1$. 

Then there exists an automorphism $f\colon (V,\nabla,\iota)\rightarrow (V,\nabla,\iota)$ such that 
	\[\Fil = f^*(\Fil').\]  
In other words, one has an isomorphism $f\colon (V,\nabla,\Fil,\iota)\rightarrow (V,\nabla,\Fil',\iota)$

\end{lem}
\begin{proof} The Hodge-de Rham spectral sequence associated to $(V,\nabla,\Fil,\varphi,\iota)_1$ degenerates at $E_1$ \cite[Lemma 6.1]{KYZ}. Then this follows from \cite[Theorem 1.6(2)]{KYZ}\footnote{While  \cite[Theorem 1.6(2)]{KYZ} is written for vector bundles with a (logarithmic) flat connection, it easily generalizes to the case with endomorphism structure}.
\end{proof}

\section{The Proof}
The proof of the main theorem of this article is diagrammatically sketched below; the definition of the various terms will follow. Here is a two sentence summary of the proof. The Langlands correspondence implies that $\tMFIsoc$ is finite. By following the diagram, it follows that $\RepGeo$ is finite.
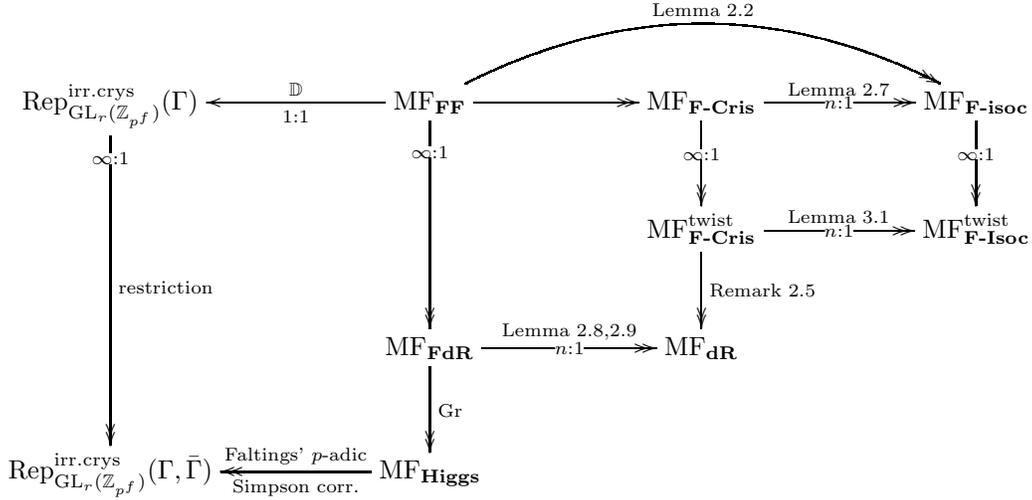
\begin{equation*}
\xymatrix@C=2cm@R=1cm{
	\Rep 
	\ar@{<-}[r]^-{\mathbb D}_-{1:1}
	\ar@{->>}[ddd]|(0.15){\infty:1}^{\text{restriction}}
	&\MF
	\ar@{->>}[r]
	\ar@{->>}[dd]|(0.2){\infty:1} 
	\ar@{->>}@/^30pt/[rr]^{\text{Lemma~\ref{irrF-Isoc}}} 
	& \MFphi  
	\ar@{->>}[d]|(0.4){\infty:1}
	\ar@{->>}[r]^{\text{Lemma~\ref{finitenessExtension}}}|{n:1} 
	&\MFIsoc \ar@{->>}[d]|(0.4){\infty:1} \\
	&& \tMFphi 
	\ar@{->>}[r]|{n:1}^{\text{Lemma~\ref{InclusionExtensions}}} 
	\ar@{->>}[d]^{\text{Remark~\ref{InvariantUnderlyingdRBundle}}}		
	&  \tMFIsoc \\ 
	&\MFFil  
	\ar@{->>}[d]^{\text{Gr}} \ar@{->>}[r]^-{\text{Lemma~\ref{finiteHodgeFil},\ref{uniquenessliftFil}}}|-{n:1} 
	& \MFnabla\\
\RepGeo  \ar@{<<-}[r]^-{\text{Faltings' $p$-adic}}_-{\text{Simpson corr.}} & \MFtheta\\
}
\end{equation*} 
We explain all of the terms in the above diagram.
\begin{itemize}
		\item $\Gamma=\pi_1^{\text{\'et}}(U_K,x)$ and $\bar \Gamma=\pi_1^{\text{\'et}}(U_{\bar K},x)$.
		\item $\Rep$ is the set of isomorphism classes of logarithmic crystalline representations $\rho\colon\Gamma\rightarrow \mathrm{GL}_r(\mathbb Z_{p^f})$ whose Hodge Tate weights are located in $[0,p-1]$ such that $\rho_{\mathbb{Q}}\colon \Gamma\rightarrow \mathrm{GL}_r(\mathbb{Q}_{p^f})$ is geometrically absolutely irreducible. 
		\item  $\RepGeo$ is the set of isomorphism classes of representations of $\bar \Gamma$ that come from $\Rep$ under the restriction map induced by the natural embedding map $\bar \Gamma \hookrightarrow \Gamma$. Thus one has a surjective map
			\[ \Rep \twoheadrightarrow \RepGeo.\]
		\item $\MF$ is the set of isomorphism classes of logarithmic Fontaine-Faltings module $(V,\nabla, \Fil,\varphi)$ with endomorphism structure $\iota\colon \mathbb{Z}_{p^f}\hookrightarrow \text{End}(V,\nabla, \Fil,\varphi)$ associated to representations in $\Rep$ via \cite{Fal89}. Thus one has an bijection
			\[ \Rep \xrightarrow{1:1} \MF.\]
		\item $\MFFil$ is the image of $\MF$ under the map that sends an isomorphism class of a logarithmic Fontaine-Faltings module $[(V,\nabla,\Fil,\varphi,\iota)]$ to the isomorphism class of the quadruple $(V,\nabla,\Fil,\iota)$ in the additive category of filtered de Rham bundles equipped with a $\mathbb{Z}_{p^f}$-endomorphism structure.  Thus one has a surjective map
			\[ \MF\twoheadrightarrow \MFFil.\]
		\item $\MFphi$ is the image of $\MF$ under the map that sends an isomorphism class of a logarithmic Fontaine-Faltings module to the isomorphism class of the underlying logarithmic $F$-crystal in locally free modules with endomorphism structure: $[(V,\nabla,\Fil,\varphi,\iota)]\mapsto[(V,\nabla,\varphi,\iota)]$. Thus one has a surjective map
			\[ \MF\twoheadrightarrow \MFphi.\]
		
		\item $\tMFphi$ is the set of equivalence classes of the set $\MFphi$ modulo the equivalence relations defined by twisting a by constant rank 1 Fontaine-Faltings modules (with endomorphism structure).
		\item $\MFnabla$ is the image of $\MFphi$ under the map that sends an isomorphism class of a logarithmic $F$-crystal to the isomorphism class of the underlying logarithmic de Rham {\it bundle} with endomorphism structure: $[(V,\nabla,\varphi,\iota)]\mapsto[(V,\nabla,\iota)]$. Thus one has surjective maps
		\[ \MFFil\twoheadrightarrow \MFnabla \twoheadleftarrow \MFphi.\]
		By Remark~\ref{InvariantUnderlyingdRBundle} the second surjective map factors through $\tMFphi$.
		
		\item  $\MFtheta$ is the image of $\MFFil$ under the map that sends an isomorphism class of filtered de Rham bundle with endomorphism structure to the isomorphism class of the associated Higgs bundle with endomorphism structure: $[(V,\nabla,\Fil,\iota)] \mapsto [\mathrm{Gr}(V,\nabla,\Fil,\iota)]$.   Thus one has a surjective map
			\[ \MFFil\twoheadrightarrow \MFtheta.\] 
		\item $\MFIsoc$ is the image of $\MF$ under the map that sends an isomorphism classs of a logarithmic Fontaine-Faltings module to the isomorphism class of the associated overconvergent $F$-isocrystal and multiplication by $\mathbb{Q}_{p^f}$, i.e., the isomorphism class of an object of $\FIsoc^{\dagger}(U_1)_{\mathbb Q_{p^f}}$.

\item $\tMFIsoc$ is the set of equivalence classes of the set $\MFIsoc$ modulo the equivalence relations defined by twisting a constant rank-1 $F$-isocrystal.

\end{itemize}
Every constant rank 1 $F$-isocrystal comes from a constant rank 1 Fontaine-Faltings module. Therefore the natural map $\tMFphi\rightarrow \tMFIsoc$ is surjective.

	If the reader is uncomfortable with carrying around the endomorphism structure $\iota$, we introduce the following notation: if   $(V,\nabla,\Fil,\varphi,\iota)$ is a logarithmic Fontaine-Faltings modules, then $(V,\nabla,\Fil)^{(0)}$ denotes the identity eigenspace of the action of $\iota$ on $(V,\nabla,\Fil)$, i.e., for any $v\in V$,
	\[v\in V^{(0)} \quad  \text{ iff }  \quad  \iota(r)v = rv \text{ for all } r\in \mathbb Z_{p^f}.\]
	One may replace all objects above with the identity eigenspaces of $\iota$ (together with $\varphi^f$, if $\varphi$ shows up). This will yield equivalent categories because $\mathbb{F}_{p^f}\subset k$ and hence $\mathbb{Z}_{p^f}\subset W(k)$.
	
	We have one final preliminary result, using the above notation.

\begin{lem}\label{InclusionExtensions} The map $$\tMFphi\twoheadrightarrow\tMFIsoc$$ is finite-to-one.
\end{lem}

\begin{proof} Fix an object $\mathcal E=(\mathcal V,\nabla,\varphi,\iota)$ in $\MFIsoc$. Recall that any twisting of $\mathcal E$ by a constant rank-1 $F$-isocrystal is of the form  $\mathcal E_\lambda = (\mathcal V,\nabla,\lambda\varphi,\iota)$ for some $\lambda\in (K\otimes \mathbb Q_{p^f})^\times$. Let $\mathcal V=\bigoplus_{i=0}^{f-1} \mathcal V_i$ be the eigen decomposition of $\iota\colon \mathbb{Q}_{p}\rightarrow \text{End}(\mathcal V,\nabla,\varphi)$. Then the semi-linear map $\varphi$ can be decomposed as semi-linear maps 
\[\varphi_i\colon \mathcal V_i \rightarrow \mathcal V_{i+1} \text{ for } i=0,\cdots,f-2 \quad \text{and} \quad \varphi_{f-1} \colon \mathcal V_{f-1} \rightarrow \mathcal V_{0}.\]
Analogously, for any $\lambda=(\lambda_i) \in (K\otimes\mathbb Q_{p^f})^\times =\prod_{i=0}^{f-1}K^\times$, the isogeny $\lambda\cdot \varphi$ decomposes as 
\[\lambda_i\varphi_i\colon \mathcal V_i \rightarrow \mathcal V_{i+1} \text{ for } i=0,\cdots,f-2 \quad \text{and} \quad \lambda_{f-1}\varphi_{f-1} \colon \mathcal V_{f-1} \rightarrow \mathcal V_{0}.\]
We denote 
\[v_p(\lambda):= \frac1f\sum_{i=0}^{f-1}v_p(\lambda_i)  \in \frac1f\mathbb Z.\]
Set 
\[\mathcal T_\lambda :=\left\{   T\in \MFphi \mid T\otimes_{\mathbb Z_p}\mathbb Q_p \simeq \mathcal E_\lambda\right\},\]
i.e., the set of $F$-crystals in finite, locally free modules which underlie $(\mathcal{V},\nabla,\varphi,\iota)$ up to isomorphism. This is a finite set by Lemma~\ref{finitenessExtension} for the following reason: if $(V,\nabla,\varphi,\iota)$ comes from a Fontaine-Faltings module, then $V$ is automatically semistable.\footnote{One way of seeing this is that a strict $p$-torsion Fontaine-Faltings module corresponds to a periodic Higgs-de Rham flow. The Higgs bundles in the flow are all semistable by \cite[Proposition 6.3]{LSZ13a}. Because $C^{-1}$ preserves semistability, this implies that $(V,\nabla)$ is semistable.} Under this notation, The lemma is claims that the following set 
\[\bigcup_\lambda \mathcal T_\lambda \Big/\sim \]
is finite, where the equivalence relation ``$\sim$'' is given as follows: $(V,\nabla,\varphi,\iota)\sim (V',\nabla',\varphi',\iota')$ if and only if they are twists by a constant, rank 1 Fontaine-Faltings module (with endomorphism structure). The finiteness will then follow from the following two claims:
\begin{itemize}
\item[Claim 1.]  $\mathcal T_\lambda = \emptyset$, if $v_p(\lambda)>p-1$ or $v_p(\lambda)<1-p$.
\item[Claim 2.] $\mathcal T_\lambda = \mathcal T_{\lambda'}$ in $\MFphi^{twist}$, if $v_p(\lambda)=v_p(\lambda')$.
\end{itemize} 

Let $(V,\nabla,\Fil,\varphi,\iota)\in \MF$ be an object mapping to $\mathcal E$. By strong divisibility, one has 
\[p^{p-1} V_{0} \subset \langle\varphi_{f-1}(V_{f-1})\rangle \subset V_{0} \quad \text{ and } \quad p^{p-1} V_{i+1} \subset \langle\varphi_i(V_i)\rangle \subset V_{i+1} \quad \text{ for } i=0,1\cdots,f-2.\]
Considering the composition, one gets
\[p^{(p-1)f}V_0 \subset \langle \varphi_{f-1}\circ\cdots\circ\varphi_{0}(V_0) \rangle 
\subset V_0.\]
Choose a basis of $V_0$ and write $\varphi_{f-1}\circ\cdots\circ\varphi_{0}$ in terms of this basis. Then the $p$-adic valuation of the determinant of this matrix is well-defined; one has
\begin{equation}\label{equ1}
0\leq v_p\Big(\det (\varphi_{f-1}\circ\cdots\circ\varphi_{0})\Big) \leq (p-1)f\cdot \rank(V_0).
\end{equation}
Suppose $\mathcal T_\lambda\neq \emptyset$ for some $\lambda=(\lambda_i)  \in (K\otimes\mathbb Q_{p^f})^\times =\prod_{i=0}^{f-1}K^\times$. Then by precisely analagous reasoning, one has
\begin{equation}\label{equ2}
0\leq v_p\Big(\det ((\lambda_{f-1}\varphi_{f-1})\circ\cdots\circ(\lambda_0\varphi_{0}))\Big) \leq (p-1)f\cdot \rank(V_0)
\end{equation}
Since $v_p\Big(\det ((\lambda_{f-1}\varphi_{f-1})\circ\cdots\circ(\lambda_0\varphi_{0}))\Big) = v_p(\lambda)\cdot f\cdot \rank(V_0) + v_p\Big(\det (\varphi_{f-1}\circ\cdots\circ\varphi_{0})\Big)$, by (\ref{equ1}) and (\ref{equ2}), one has 
\[1-p\leq v_p(\lambda) \leq p-1.\]
Thus the Claim 1 follows.

We now show Claim 2. By replacing $\varphi$ with $\lambda'\varphi$, one may reduce the claim to case $\lambda'=1$; in this setting, as $v_p(\lambda)=v_p(\lambda')$, it follows that $v_p(\lambda)=0$. Denote $n_i = v_p(\lambda_i)$, $c_i = \lambda_i/p^{n_i}\in W^\times$ and $m_{i} = n_1 + n_2 +\cdots + n_{i-1}$ for all $i=0,1,\cdots,f-1$. Consider the following map
\[F_{1,\lambda}\colon \mathcal T_1\rightarrow \mathcal T_\lambda\]
which maps $(\oplus_iV_i, \oplus_i\nabla_i,\oplus_i \varphi_i,\iota)$ to $(\oplus_i p^{m_i}V_i, \oplus_i\nabla_i,\oplus_i\lambda_i\varphi_i,\iota)$. The Claim 2 is reduced to showing that this map
\begin{itemize}
\item[(1)] is well-defined;
\item[(2)] is bijective and 
\item[(3)] preserves the twisted classes.
\end{itemize}

We first show it is well-defined. Suppose $(\lambda_i)$ are $K^{\times}$ with $\sum_{i=0}{f-1}v_p(\lambda_i)=0$. Set $n_i$, $c_i$, and $m_i$ as above. Suppose $\Fil_i$ is a Hodge filtration on $(V_i,\nabla_i)$ such that $(\oplus_iV_i, \oplus_i\nabla_i,\oplus_i\Fil_i,\oplus_i \varphi_i,\iota)$ forms an object in $\MF$. Then $(\oplus_i p^{m_i}V_i, \oplus_i\nabla_i,\oplus_i\Fil_i,\oplus_i\lambda_i\varphi_i,\iota)$ is also contained in $\MF$, because the pair $(\oplus_i\Fil_i,\oplus_i\lambda_i\varphi_i)$ also satisfies strong divisibility on $\oplus_i p^{m_i}V_i$:
\[(\lambda_i\varphi_i)(p^{m_i}\widetilde{V_i}) = \lambda_ip^{m_i} \varphi_i(\widetilde{V_i}) = \lambda_ip^{m_i}V_{i+1} = p^{m_{i+1}}V_{i+1}\]
and
\[(\lambda_{f-1}\varphi_{f-1})(p^{m_{f-1}}\widetilde{V_{f-1}}) = \lambda_{f-1}p^{m_{f-1}} \varphi_{f-1}(\widetilde{V_{f-1}}) = \lambda_{f-1}p^{m_{f-1}}V_{0} = p^{m_0}V_{0},\]
in the last equality we used the fact that $n_{f-1}+m_{f-1}=f\cdot v_p(\lambda)=0=m_0$. Thus $F_{1,\lambda}$ is well-defined. The map $F_{1,\lambda}$ is bijective, because one can define its inverse map $\mathcal T_\lambda\rightarrow \mathcal T_1$ in similar manner by sending $(\oplus_iV_i, \oplus_i\nabla_i,\oplus_i \varphi_i,\iota)$ to $(\oplus_i p^{-m_i}V_i, \oplus_i\nabla_i,\oplus_i\lambda_i^{-1}\varphi_i,\iota)$. 

Now we only need to show  $(\oplus_iV_i, \oplus_i\nabla_i,\oplus_i \varphi_i,\iota)$ and $(\oplus_i p^{m_i}V_i, \oplus_i\nabla_i,\oplus_i\lambda_i\varphi_i,\iota)$ differ by twisting a rank 1 Fontaine-Faltings module. By the commutativity of following diagram
\begin{equation*}
\xymatrix{
V_0 \ar[d]^{p^{m_0}\cdot\mathrm{id}} \ar@{_(->}[r] & \widetilde{V_0} \ar[d]^{p^{m_0}\cdot\mathrm{id}} \ar[r]^{\simeq}_{\varphi_0}
& 
V_1 \ar[d]^{p^{m_1}\cdot\mathrm{id}} \ar@{_(->}[r]  & \widetilde{V_1} \ar[d]^{p^{m_1}\cdot\mathrm{id}} \ar[r]^{\simeq}_{\varphi_1}
& 
V_2 \ar@{..}[r] \ar[d]^{p^{m_2}\cdot\mathrm{id}}
& 
V_{f-1} \ar[d]^{p^{m_{f-1}}\cdot\mathrm{id}}  \ar@{_(->}[r]  & \widetilde{V_{f-1}} \ar[d]^{p^{m_{f-1}}\cdot\mathrm{id}} \ar@/_30pt/[llllll]^{\simeq}_{\varphi_{f-1}}\\
p^{m_0}V_0 \ar@{^(->}[r] & p^{m_0}\widetilde{V_0} \ar[r]^{\simeq}_{p^{n_0}\cdot\varphi_0}
& 
p^{m_1}V_1 \ar@{^(->}[r]  & p^{m_1}\widetilde{V_1} \ar[r]^{\simeq}_{p^{n_1}\cdot\varphi_1}
& 
p^{m_2}V_2 \ar@{..}[r]
& 
p^{m_{f-1}}V_{f-1} \ar@{^(->}[r]  & p^{m_{f-1}}\widetilde{V_{f-1}} \ar@/^30pt/[llllll]^{\simeq}_{p^{n_{f-1}}\cdot\varphi_{f-1}}\\
}
\end{equation*}
one gets an isomorphism $\oplus_i(p^{m_i}\mathrm{id}) \colon (\oplus_iV_i, \oplus_i\nabla_i,\oplus_i \varphi_i,\iota) \rightarrow (\oplus_i p^{m_i}V_i, \oplus_i\nabla_i,\oplus_ip^{n_i}\varphi_i,\iota)$. 
We consider the rank-1 Fontaine-Faltings module with endomorphism structure
\[\mathcal L =\Big(\bigoplus_{i=0}^f W\cdot e_i,\Fil_{tri}, \phi,\iota\Big)\]
where $\phi(\sum_i a_ie_i) = c_0a_0^\sigma e_1 + \cdots  + c_{f-2}a_{f-2}^\sigma e_{f-1}+ c_{f-1}a_{f-1}^\sigma e_0$. Since $v_p(c_i)=0$ for all $i=0,1,\cdots,f-1$, this is a well-defined constant Fontaine-Faltings module. Then consider the twisting of $(\oplus_i p^{m_i}V_i, \oplus_i\nabla_i,\oplus_ip^{n_i}\varphi_i,\iota)$ by $\mathcal L$, which is nothing just equal to $(\oplus_i p^{m_i}V_i, \oplus_i\nabla_i,\oplus_i\lambda_i\varphi_i,\iota)$.
Thus  $(\oplus_iV_i, \oplus_i\nabla_i,\oplus_i \varphi_i,\iota)$ and $(\oplus_i p^{m_i}V_i, \oplus_i\nabla_i,\oplus_i\lambda_i\varphi_i,\iota)$ are differed by twisting a rank-1 Fontaine-Faltings module.
\end{proof}

\begin{proof}[Proof of Theorem~\ref{mainThm}]

By Faltings' definition of a crytalline representation, the Hodge-Tate weights are in an interval of length $p-1$. Since Tate twisting of a crystalline representation $\rho$ does not change the isomorphism class of $\rho\mid_{\bar{\Gamma}}$, it suffices to prove the finiteness of set 
\[\left.\left\{\rho\colon \pi^{et}_1(U_{K}) \rightarrow \mathrm{GL}_r(\mathbb Z_{p^f}) \left| \begin{array}{l}
			\rho \text{ is log crystalline }\\
			\text{with HT weights in }[0,p-1] \text{ and } \\
			\rho^{\text{geo}}\colon \pi_1^{\text{\'et}}(U_{\bar K})\rightarrow \mathrm{GL}_r(\mathbb{Q}_{p^f}) \text{ absolutely irreducible.}\\
			\end{array}  \right.\right\}
\right/\rho_1\sim\rho_2\text{ if }\rho_1|_{\pi_1^{\text{\'et}}(U_{\bar K})}\cong \rho_2|_{\pi_1^{\text{\'et}}(U_{\bar K})}\]

Equivalently, to prove Theorem~\ref{mainThm} we will show that $\RepGeo$ is finite.

Since the Faltings $p$-adic Simpson's correspondence~\cite{Fal05} is compatible with his $\mathbb D$-functor~\cite{Fal89}, one has following commutative diagram of surjective maps between sets
\begin{equation*}
\xymatrix@C=4cm{ 
\MF \ar@{->>}[r]^{(V,\nabla,\Fil,\varphi,\iota) \mapsto \mathrm{Gr}\big((V,\nabla,\Fil,\iota)\big)} \ar[d]_{\simeq}^{\mathbb D}		
& \MFtheta \ar[d]^{\text{Faltings' $p$-adic Simpson correspondence}} \\
\Rep \ar@{->>}[r]^{\text{restriction}} 
& \RepGeo \\
}
\end{equation*}
	
Since the two horizontal arrows and the left vertical arrow are surjective, the right vertical arrow is also surjective. One has a surjective composition
\[ \MFFil\twoheadrightarrow \MFtheta\twoheadrightarrow  \RepGeo.\]
To prove Theorem~\ref{mainThm}, we only need to show the finiteness of $\MFFil$.

Firstly, we claim that $\tMFIsoc$ is finite.  By Lemma~\ref{absolutely_irreducible}, all elements in $\MFIsoc$ are of absolutely irreducible. Then the set of equivalence classes of absolutely irreducible objects of $\FIsoc^{\dagger}(U_1)_{\mathbb Q_{p^f}}$ up to twisting by a constant rank 1 $F$-isocrystal is finite by \cite[Corollary 2.1.5]{kedlayacompanions}.
 
Secondly, we claim that $\tMFphi$ is finite. This follows from Lemma~\ref{InclusionExtensions}.

Finally, the map $\MFFil\twoheadrightarrow\MFnabla$ is finite-to-one by Lemma~\ref{finiteHodgeFil} and Lemma~\ref{uniquenessliftFil}. Thus $\MFFil$ is finite. 
\end{proof}

 \end{document}